%% file: main.tex
\documentclass[sn-mathphys,Numbered]{sn-jnl}%

\usepackage{graphicx}%
\usepackage{multirow}%
\usepackage{amsmath,amssymb,amsfonts}%
\usepackage{amsthm}%
\usepackage{mathrsfs}%
\usepackage{mathtools}
\usepackage[title]{appendix}%
\usepackage{xcolor}%
\usepackage{textcomp}%
\usepackage{manyfoot}%
\usepackage{booktabs}%
\usepackage{algorithm}%
\usepackage{algpseudocode}%
\usepackage{listings}%
\usepackage{pifont}

\newtheorem{theorem}{Theorem}
\newtheorem{example}{Example}%
\newtheorem{lemma}{Lemma}%
\newtheorem{corollary}{Corollary}[theorem]
\newtheorem{definition}{Definition}%
\newtheorem{assumption}{Assumption}

\newcommand*{\R}{{\mathbb R}}

\DeclarePairedDelimiter{\ceil}{\lceil}{\rceil}

\newcommand{\avp}[1]{\textcolor{black}{#1}} 

\graphicspath{{figures/}}
\input{header}

\begin{document}

\title[Decentralized Optimization Over Slowly Time-Varying Graphs: Algorithms and Lower Bounds
]{Decentralized Optimization Over Slowly Time-Varying Graphs: Algorithms and Lower Bounds}

\author[1]{\fnm{Dmitry} \sur{Metelev}}\email{metelev.ds@phystech.edu}

\author[1,3]{\fnm{Aleksandr} \sur{Beznosikov}}\email{beznosikov.an@phystech.edu}

\author[1]{\fnm{Alexander} \sur{Rogozin}}\email{aleksandr.rogozin@phystech.edu}


\author[1,2,4]{\fnm{Alexander} \sur{Gasnikov}}\email{gasnikov@yandex.ru}

\author[5]{\fnm{Anton} \sur{Proskurnikov}}\email{anton.p.1982@ieee.org}

\affil[1]{\orgname{Moscow Institute of Physics and Technology}, \orgaddress{\city{Moscow}, \country{Russia}}}

\affil[2]{\orgname{ISP RAS Research Center for Trusted Artificial Intelligence}, \orgaddress{\city{Moscow}, \country{Russia}}}

\affil[3]{\orgname{Mohamed bin Zayed University of Artificial Intelligence}, \orgaddress{\city{Abu Dhabi}, \country{United Arab Emirates}}}



\affil[4]{\orgname{Skolkovo Institute of Science and Technology}, \orgaddress{\city{Moscow}, \country{Russia}}}

\affil[5]{\orgname{Politecnico di Torino}, \orgaddress{\city{Turin}, \country{Italy}}}


\abstract{
We consider a decentralized convex unconstrained optimization problem, where the cost function can be decomposed into a sum of strongly convex and smooth functions, associated with individual agents, interacting over a static or time-varying network. Our main concern is the convergence rate of first-order optimization algorithms as a function of the network's graph, more specifically, of the condition numbers of gossip matrices. We are interested in the case when the network is time-varying but the rate of changes is restricted. We study two cases: randomly changing network satisfying Markov property and a network changing in a deterministic manner. For the random case, we propose a decentralized optimization algorithm with accelerated consensus. For the deterministic scenario, we show that if the graph is changing in a worst-case way, accelerated consensus is not possible even if only two edges are changed at each iteration. The fact that such a low rate of network changes is sufficient to make accelerated consensus impossible is novel and improves the previous results in the literature. 



}

\maketitle


\section{Introduction}

The purpose of this paper is to study the problem of \avp{distributed unconstrained optimization problem, where the cost function is constructed as the average}
\begin{equation}\label{main_problem}
    \min_{x\in \R^m}~f(x) = \frac{1}{n}\sum_{i=1}^n{f_i(x)},
\end{equation}
\avp{of $n$ strongly convex functions} $\{f_i\}_{i=1}^n$, \avp{associated to $n$ autonomous agents.} 

\avp{Following the standard framework of distributed convex optimization \cite{nedic2009distributed,nedic2009subgradient,scaman2017optimal}, we assume that agents communicate synchronously and can transmit real numbers of vectors to their teammates; the effects of communication delays and packet losses are ignored. At each iteration of the algorithm agent $i$ updates its state by applying some first-order\footnote{The exact definition of the first-order optimization algorithm will be given below.} algorithm aiming at minimizing of the function $f$; this algorithm can use the internal variables of agent $i$ and information obtained from some of the other agents through a communication network. This communication network is represented by a graph whose vertices (nodes) are in one-to-one correspondence with the agents and whose edges represent communication channels available at the current iteration.} In various settings, the communication network may remain static or change in a certain way, in our case in particular, we study the problem \eqref{main_problem}, imposing specific constraints on the change of the communication network.





Decentralized optimization has emerged as an essential tool for managing sum-type problems of type \eqref{main_problem}. Decentralized algorithms have found significant applications in areas where centralized coordination is limited due to data volume or privacy restrictions. Agents in these decentralized systems maintain local optimization objectives and participate in a network whose structure may evolve over time. Such problems find application in wireless sensor networks \cite{Iacca_2013}, resource allocation problems \cite{9901830}, distributed averaging \cite{cai2014average,XIAO200465}, distributed sensing \cite{bazerque2009distributed}, vehicle coordination and control \cite{ren2008distributed}, formation control \cite{olshevsky2010efficient,ren2006consensus,jadbabaie2003coordination}, distributed data analysis \cite{rabbat2004distributed,forero2010consensus,nedic2017fast}, power system control \cite{ram2009distributed,gan2012optimal}.


\vspace{0.3cm}
\noindent\textbf{Related Work}. In the literature, the complexity of decentralized optimization algorithms is typically represented by condition number of the network $\chi$ and condition number of objective functions $\kappa=L/\mu$. For the case of a static network, this problem is relatively well-studied. In the work \cite{scaman2017optimal} a communication complexity lower bound of $\Omega\left(\sqrt{\chi\kappa}\log\left(\frac{1}{\epsilon}\right)\right)$ was established and the optimal algorithm called MSDA was proposed, assuming access to the dual oracle. In the case of the primal oracle, the optimal algorithm OPAPC \cite{kovalev2020optimal} was suggested, reaching lower bounds from \cite{scaman2017optimal}.

In the non-static case, when the network can arbitrarily change over time, a lower bound of $\Omega\left(\chi\sqrt{\kappa}\log\left(\frac{1}{\epsilon}\right)\right)$ was established in \cite{kovalev2021lower}. Corresponding optimal algorithms were also derived: ADOM+ \cite{kovalev2021lower}, Acc-GT \cite{li2021accelerated}, considering the primal oracle, and ADOM \cite{kovalev2021adom}, considering the dual oracle.

Regarding the lower bounds in the case of a slowly time-varying network (when constraints are imposed on its rate of change), \cite{metelev2023consensus} obtained three different lower bounds, each depending on the degree of constraint on the rate of edge changes per temporal iteration. Specifically, they correspond to the following regimes:

\noindent $\bullet$ The mode with $\mathcal{O}(n^{\alpha})$ ($\alpha>0$) edge changes yields a lower bound of $\Omega(\chi\sqrt{\kappa}\log{\frac1{\epsilon}})$.

\noindent $\bullet$ The mode with $\mathcal{O}(\log(n))$ edge changes corresponds to $\Omega(\frac{\chi}{\log{\chi}}\sqrt{\kappa}\log{\frac1{\epsilon}})$.

\noindent $\bullet$ The mode with $c=const$ edge changes corresponds to $\Omega(\chi^{d(c)}\sqrt{\kappa}\log{\frac1{\epsilon}})$, where $c\ge12$ and $\frac12 < d(c) < 1$.


\vspace{0.3cm}
\noindent\textbf{Our contribution.} The contribution of this paper is twofold.

Firstly, we study consensus algorithms over time-varying graphs with restricted changes that change randomly and satisfy Markov condition. We treat consensus problem as a stochastic optimization problem and propose an accelerated consensus method that is based on accelerated stochastic gradient method. After that, we propose an accelerated method for decentralized optimization under our assumptions.

Secondly, we show that accelerated consensus is not attainable for decentralized optimization over time-varying networks wit worst-case changes. Our lower bounds are based on a counterexample graph in which no more than two edges are altered at each iteration. Previously lower bounds were provided in \cite{metelev2023consensus}, and our results make a significant improvement over that.

\section{Preliminaries}

\subsection{Smoothness and strong convexity}

In this paper, by $H$ we denote any Hilbert space over $\R$, such as $\R^n$ or $\ell_2$.
\begin{definition}[$\mu$-Strongly Convex Function]
A function $h: H\to\R$ is called $\mu$-strongly convex if for any $x, y\in H$, the following inequality holds
\begin{equation*}
h(y) \geq h(x) + \langle \nabla h(x), y - x \rangle + \frac{\mu}{2}\|y - x\|^2.
\end{equation*}
\end{definition}

\begin{definition}[$L$-Smooth Function]
A function $h: H\to\R$ is called $L$-smooth if for any $x, y\in H$, it satisfies
\begin{equation*}
\|\nabla h(y) - \nabla h(x)\|_*\leq L\|y - x\|.
\end{equation*}
\end{definition}


We will refer to the functions \avp{$f_i$} at the nodes of the network as \textit{local functions}. The function \avp{$f$} in \eqref{main_problem} will be referred to as \textit{the global function}.

\subsection{Laplacians}
Further on in the paper, we consider only loop-less undirected graphs.

\begin{definition}[Weighted Graph]
    Let $\cG_A = (\cV, \cE, A)$ denote an undirected weighted graph with nodes $\cV$, edges $\cE$ and edge weights represented by adjacency matrix $A = [a_{ij}]_{i,j=1}^n$. Weight $a_{ij}$ is positive if $(i, j)\in\cE$ and zero otherwise.
\end{definition}

\begin{definition}[Laplacian of a Weighted Graph]
	Let $\cG_A = (\cV, \cE, A)$ be a weighted graph. The Laplacian of $\cG_A$ is defined as
	\begin{align*}
		[L(\cG)]_{ij} = 
		\begin{cases}
			\sum_{(k, i)\in\cE} a_{ik}, &\text{if } i = j, \\
			-a_{ij}, &\text{if } (i, j)\in\cE, \\
			0, &\text{else}.
		\end{cases}
	\end{align*}
\end{definition}

The unweighted or standard Laplacian of unweighted graph $\cG = \cG(\cV, \cE)$ is simply the weighted Laplacian of the weighted graph $\cG_A = (\cV, \cE, A)$ with all weights set to $1$, i.e. $a_{ij} = 1$ if $(i, j)\in\cE$ and $a_{ij} = 0$ if $(i, j)\notin\cE$.

\noindent\textbf{Example:} Consider a graph $\cG_A = (\cV, \cE, A)$ with 3 vertices $\cV = (1, 2, 3)$ and edges $\cE = \braces{(1, 2), (2, 3)}$ with weights $a_{12} = 2$ and $a_{23} = 1$. The weighted Laplacian matrix for this graph is given by:
$$
L(\cG_A) =
\begin{bmatrix}
	2 & -2 & 0 \\
	-2 & 3 & -1 \\
	0 & -1 & 1
\end{bmatrix}.
$$

It is worth mentioning the well-known
\begin{lemma}
\avp{For a weighted graph $\cG_A$ with positive weights $a_{ij}>0\,~\forall (i,j)\in \cE$, the Laplacian $L(\cG_A)$ is a positive} semidefinite symmetric matrix \avp{whose kernel contains the column of ones. Furthermore, $\ker L(\cG_A) = \text{span}\{1\}$ if and only if the graph $\cG_A$ is   connected.}
\end{lemma}
For the proof, see \cite[Lemma~2]{Murray:07}.

Moreover, we introduce a mini-laplacian
\begin{definition}
    Let $\cG = (\cV, \cE)$ be a graph. The mini-Laplacian $\ell_{ij}$ is an $|\cV| \times |\cV|$ matrix defined as follows:
\begin{align}\label{mini_laplacian}
	[\ell_{ij}]_{kl} = \begin{cases}
		{~~}1, & \text{if } (k, l) = (i, i) \text{ or } (k, l) = (j, j), \\
		-1, &\text{if } (k, l) = (i, j) \text{ or } (k, l) = (j, i), \\
		{~~}0, & \text{otherwise}.
	\end{cases}
\end{align}
\end{definition}

\subsection{Gossip matrices}

\begin{definition}[Gossip Matrix]\label{definitions_gossip}
	Let $\cG = (\cV, \cE)$ be a graph with $n$ nodes. We call matrix $W(\cG)\in\R^{n\times n}$ a gossip matrix if
	\begin{enumerate}
		\item $[W(\cG)]_{i,j} = 0$, if $i \neq j$ and $(i, j) \notin \cE$.
		\item $\ker W(\cG) = \{(x_1, \dots, x_n) \in \mathbb{R}^n : x_1 = \dots = x_n\}$.
	\end{enumerate}
\end{definition}

%
%



\noindent Note that a Laplacian of a weighted graph satisfies Definition~\ref{definitions_gossip}.


\section{Upper bounds}

In this section, we show that one can organize an accelerated consensus procedure for communication networks changing slowly and according to Markovian law. Using this procedure, we can achieve an improvement in the number of communications in decentralized optimization algorithms.

\subsection{Consensus for networks with Markovian changes}

Since communication networks change over time, the gossip matrices corresponding to these networks also time-varying. We define $G$ as the set of all possible graphs that can occur through time and $W_G$ as as the set of gossip matrices for $G$. For simplicity we can consider that each graph $\mathcal{G} \in G$ corresponds to exactly one matrix from $W \in W_G$. But one can note that for the graph $\mathcal{G}$ it is possible to define different gossip matrices at different moments of time (depending on the needs), therefore in general $|W_G| \geq | G |$. This case is also suitable for further reasoning and analysis.

Let us also introduce additional properties of graph change. In particular, we assume that the sequence of gossip matrices $\{W(\mathcal{G}_i)\}_{i=0}^{\infty}$ is a time-homogeneous Markov chain. We define $W_{\sigma}$ as $\sigma$-field on $W_G$. We also denote by $\MKQ$ the corresponding Markov kernel and impose the following assumption on the mixing properties of $\MKQ$:
\begin{assumption}
	\label{as:Markov_chain}
	$\{W(\mathcal{G}_k)\}_{k=0}^{\infty}$ is a stationary Markov chain on $(W_G,W_{\sigma})$ with Markov kernel $\MKQ$ and unique invariant distribution $\pi$. Moreover, $\MKQ$ is uniformly geometrically ergodic with mixing time $\taumix \in \nset$, i.e., for every $m \in \nset$,
	\begin{equation*}
		\Delta(\MKQ^m) = \sup_{W,W' \in W_G} (1/2) \norm{\MKQ^m(W, \cdot) - \MKQ^m(W',\cdot)}_{\mathsf{TV}} \leq (1/4)^{\lfloor m / \taumix \rfloor}\,.
	\end{equation*}
\end{assumption}
We also assume that
\begin{assumption}
	\label{as:unbiased}
	For all $k \in \nset \cup \{0\}$, it holds $\mathbb{E}_{\pi}[W(\mathcal{G}_k)] = \tilde W$.
\end{assumption}
The matrix $\tilde W$ is, in some sense, the keystone for the sequence $\{W(\mathcal{G}_k)\}_{k=0}^{\infty}$.
Therefore, we ficus on it and introduce some properties of $\tilde W$. In particular, we assume that
\begin{assumption}
    \label{as:w_0}
    The matrix $\tilde W$ satisfies Definition \ref{definitions_gossip}, i.e. there exists undirected connected graph  $\mathcal{\tilde  G}$ such that $\tilde W$ is a gossip matrix of $\mathcal{\tilde  G}$.
\end{assumption}
For the sake of brevity let us introduce:
\begin{align*}
	\lambda_{\max} &= \lambda_{\max}(\tilde W),~ \lambda_{\min}^+ = \lambda_{\min}^+(\tilde W),~
	\chi = \frac{\lambda_{\max}(\tilde W)}{\lambda_{\min}^+(\tilde W)}.
\end{align*}
Finally, we make the following assumption:
\begin{assumption} 
	\label{as:rho_diff}
	For any graph $\mathcal{G}$ of the set $G$ it holds:
	\begin{equation*}
		\| W (\mathcal{G}) - \tilde W \| \leq \rho.
	\end{equation*}
\end{assumption}
To understand what value $\rho$ can take, let us consider the following example.
\begin{example}
	Let us take gossip matrix of the graph as its Laplacian, i.e. $W(\cG) = L(\cG)$. Also assume that $\cG$ and $\cG'$ differ in no more than $\Delta$ edges, i.e. $|(\cE\setminus\cE')\cup (\cE'\setminus\cE)|\leq \Delta$. Then we have
	\begin{align*}
		W(\cG) - W(\cG') &= \sum_{(i, j)\in\cE\setminus\cE '} \ell_{ij} - \sum_{(i, j)\in\cE '\setminus\cE}\ell_{ij},
    \end{align*}
    where $\ell_{ij}$ denotes the mini-Laplacian defined in \eqref{mini_laplacian}. Note that $\norm{\ell_{ij}}\leq 2$. We have
    \begin{align*}
		\norm{W(\cG) - W(\cG')} &\leq \sum_{(i, j)\in (\cE\setminus\cE')\cup (\cE '\setminus\cE)}\norm{\ell_{ij}}\leq 2\Delta.
	\end{align*}
	Therefore, Assumption~\ref{as:rho_diff} holds with $\rho = 2\Delta$.
\end{example}
This example shows that $\rho$ can be is proportional to the number of distinct edges in graphs.
With Assumption \ref{as:rho_diff}, one can prove that for any $x \in \R^{n}$
\begin{eqnarray}
	\label{eq:noise_as}
	\| W (\mathcal{G}) x - \tilde Wx \| = \norm{\left(W (\mathcal{G})- \tilde W \right) (x - x^*)} \leq \rho \norm{x - x^*},
\end{eqnarray}
where $x^*_{(i)} = \frac{1}{n} \sum_{j=1}^n x_{(j)}$ for $i = 1, \ldots, n$.

Based on the $W_0$ matrix, we write down the consensus search problem:
\begin{equation}
	\label{eq:consensus_problem}
	\begin{split}
		\min_{x \in \R^{n}} & \left[r(x) = \| \sqrt{\tilde W} x\|^2\right]
		\\
		\text{s.t.} & \sum_{j=1}^n x_{(j)} = \sum_{j=1}^n x^0_{(j)}
	\end{split},
\end{equation}
where $x$ is a vector of local variables, $x^0$ is an initial vector of local variables. Here we consider that locally each device stores a scalar variable, it is clear that the result can be easily generalize to vectors of local variables.

For the problem \eqref{eq:consensus_problem}, we can apply Algorithm 1 from \cite{beznosikov2023first} (for convenience, we list it here -- see Algorithm \ref{alg:AGD_ASGD}), which is designed to solve stochastic optimization problems with Markovian nature of randomness. The essence of this method is the use of an unusual random batches (lines \ref{line_acc_3}-\ref{line_acc_4}). Note that to calculate such $g^k$ it is necessary to communicate $2^{J_{k}} B$ times in a row, but send the same values from vector $x^k_g$. Then it is possible not to additionally send values of $x^k_g$ to a neighbor, which has already been communicated with before.

In terms of convergence we can use Theorem 1 from \cite{beznosikov2023first}: the target function of \eqref{eq:consensus_problem} is $\lambda_{\max}$-smooth, Assumptions \ref{as:Markov_chain}, \ref{as:unbiased}, \ref{as:rho_diff} plunges us in the setting of A3-4 of \cite{beznosikov2023first}. But there are also problems that need to be solved. In particular, we need to deal with the fact that the target function from \eqref{eq:consensus_problem} is not strongly convex on $\ker \tilde W $. 
The key problem is that A4 of \cite{beznosikov2023first} uses that $\norm{\nabla F(x, z) - \nabla f(x)}^{2} \leq \sigma^{2} + \rho^{2} \| \nabla f(x) \|^{2}$, in our case (see \eqref{eq:noise_as}), we have $\norm{\nabla F(x, z) - \nabla f(x)}^{2} \leq \rho^{2} \| x - x^* \|^{2}$, then we need to modify the proof of Theorem 1 from \cite{beznosikov2023first}. 

\begin{algorithm}[h!]
		\caption{Accelerated consensus over graphs with Markovian changes}
		\label{alg:AGD_ASGD}
		\begin{algorithmic}[1]
			\State {\bf Parameters:} stepsize $\gamma>0$, momentums $\theta, \eta, \beta, p$, number of iterations $\nbiter$, batchsize limit $\batchbound$
			\State {\bf Initialization:} choose  $x^0_f  = x^0$, $T^0 = 0$, set the same random seed for generating $\{J_k\}$ on all devices 
			\For{$k = 0, 1, 2, \dots, \nbiter-1$}
			\State $x^k_g = \theta x^k_f + (1 - \theta) x^k$ \label{line_acc_1}
			\State Sample $\textstyle{J_k \sim \text{Geom}\left(1/2\right)}$ \label{line_acc_3}
			\State Send $x_g^k$ to neighbors in the networks $\{\mathcal{G}_{T^{k} + i}\}^{2^{J_{k}} B}_{i=1}$
			\State Compute \label{line_acc_4}
			\text{\small{ 
					$g^{k} = g^{k}_0 +
					\begin{cases}
						\textstyle{2^{J_k} \left( g^{k}_{J_k}  - g^{k}_{J_k - 1} \right)}, & \text{ if } 2^{J_k} \leq \batchbound \\
						0, & \text{ otherwise}
					\end{cases}
					$}}
			\Statex \hspace{0.4cm} with  
			\text{\small{$
					\textstyle{g^k_j = 2^{-j} B^{-1} \sum\nolimits_{i=1}^{2^j B} W (\mathcal{G}_{T^{k} + i}) x^{k}_g}
					$
			}}
			\State    $\textstyle{x^{k+1}_f = x^k_g - p\gamma g^k}$ \label{line_acc_2}
			\State    $\textstyle{x^{k+1} = \eta x^{k+1}_f + (p - \eta)x^k_f + (1- p)(1 - \beta) x^k +(1 - p) \beta x^k_g}$ \label{line_acc_3}
			\State    $\textstyle{T^{k+1} = T^{k} + 2^{J_{k}} B}$\label{line_counter}
			\EndFor
		\end{algorithmic}
\end{algorithm}

\begin{theorem}\label{th:conv_markov_acc}
	Let Assumptions \ref{as:Markov_chain}, \ref{as:unbiased}, \ref{as:w_0}, \ref{as:rho_diff} hold. Let problem \eqref{eq:consensus_problem} be solved by Algorithm~\ref{alg:AGD_ASGD}. Then for any $b \in \nset$,  
	$$
	\gamma \in \left(0; \min\left\{\tfrac{3}{4\lambda_{\max}} ; \tfrac{\lambda_{\min}^3}{[1800 \rho^2 \left(\taumix b^{-1}  + \taumix^2 b^{-2}\right)]^{2}}\right\}\right),
	$$
	and $\beta, \theta, \eta, p, \batchbound, B$ satisfying 
	\begin{eqnarray*}
		&p = \tfrac{1}{4}, \quad \beta = \sqrt{\frac{4 p^2 \mu \gamma}{3}}, \quad\eta = \frac{3\beta}{p \mu \gamma} = \sqrt{\frac{12}{\mu \gamma}}, \quad  \theta = \frac{p \eta^{-1} - 1}{\beta p \eta^{-1} - 1},
		\\
		&M = \max\{2; \sqrt{\frac{1}{4} \left(1 + \tfrac{2}{\beta}\right)}\}, \quad B = \lceil b \log_2 \batchbound \rceil,
		\,
	\end{eqnarray*}
	it holds that
	\begin{align*}
		\mathbb{E}\Bigg[&\|x^{N} - x^*\|^2 + \frac{24}{\lambda_{\min}} (r(x^{N}_f) - r(x^*)) \Bigg]
		\\
		&= \mathcal{O}\left(  
		\exp\left( - N\sqrt{\frac{p^2 \lambda_{\min} \gamma}{3}}\right) \left[\| x^0 - x^*\|^2 + \frac{24}{\lambda_{\min}} (r(x^0) - r(x^*)) \right]\right)\,,
	\end{align*}
where $x^*_{(i)} = \frac{1}{n} \sum_{j=1}^n x_{(j)}$ for $i = 1, \ldots, n$..
\end{theorem}
The proof of the theorem are given further in Section \ref{sec:upper_proofs}. From Theorem \ref{th:conv_markov_acc} immediately follows the next corollary.
\begin{corollary}
	\label{cor:complexity_accelerated}
	Under the conditions of Theorem \ref{th:conv_markov_acc}, choosing $b = \taumix$ and $\gamma \simeq \min\left\{\frac{1}{\lambda_{\max}} ; \frac{\lambda_{\min}^3}{\rho^4}\right\},
	$
	in order to achieve $\varepsilon$-approximate solution (in terms of $\EE[\|x - x^*\|^2] \lesssim \varepsilon$) it takes 
	\begin{equation*}
		\mathcal{\tilde O} \left( \taumix \left[ \sqrt{\chi} + \frac{\rho^2}{\lambda_{\min}^2}\right] \log \frac{1}{\varepsilon}\right) ~~ \text{communications}\,.
	\end{equation*}
\end{corollary}

\subsection{Decentralized optimization with new consensus procedure}

Based on Algorithm \ref{alg:AGD_ASGD}, it is possible to develop a decentralized algorithm for solving the distributed optimization problem \eqref{main_problem}.  The essence of the approach is to use the classical non-distributed algorithm. One can adapt it to a decentralized setup by applying a consensus procedure to the full global gradient calculations. In particular, we take the classical optimal method for smooth convex optimization problems -- the accelerated gradient method \cite{nesterov2003introductory} (Algorithm \ref{alg:2}). At each iteration of Algorithm~\ref{alg:2}, Algorithm~\ref{alg:AGD_ASGD} is applied when the nodes exchange local gradients with each other (line \ref{lin_alg2:1}). This approach does not achieve exact consensus, but by making a sufficient number of iterations $T$ it is possible to obtain $v^k_{i_1} \approx v^k_{i_2}$ with high accuracy. 

\begin{algorithm}[h!]
		\caption{Accelerated gradient algorithm for graphs with Markovian changes}
		\label{alg:2}
		\begin{algorithmic}[1]
			\State {\bf Parameters:} stepsize $\gamma>0$, momentums $\eta$, number of iterations $\nbiter$, number of communications $T$
			\State {\bf Initialization:} choose  $y^0_i  = x^0_i = x^0$
			\For{$k = 0, 1, 2, \dots, \nbiter-1$}
            \State Locally compute $\nabla f_i (y^k_i)$
            \State Communicate by running $T$ iterations of Algorithm \ref{alg:AGD_ASGD}  \label{lin_alg2:1}
            \Statex \hspace{0.4cm} with initialization $\{\nabla f_i (y^k_i)\}_{i=1}^n$ and output $\{v^k_i\}_{i=1}^n$
			\State Locally make update: $x^{k+1}_i = y^k_i - \gamma v^k_i$ 
			\State Locally make update: $y^{k+1}_i = x^{k+1}_i + \eta (x^{k+1}_i - x^k_i)$
			\EndFor
		\end{algorithmic}
\end{algorithm}

The analysis of this kind of algorithms is technical, namely, one need to add small inexactness to the analysis of the basic non-distributed method \cite{beznosikov2020distributed, rogozin2021towards, rogozin2021accelerated, beznosikov2021near}. If we want to solve the optimization problem \eqref{main_problem} with precision $\varepsilon$, then by requiring consensus from Algorithm \ref{alg:AGD_ASGD} to precision $\varepsilon^2$ or $\varepsilon^3$, we do not feel the effect of consensus inexactness.  And therefore the following corollary holds.

\begin{corollary}
Let the function $f$ from \eqref{main_problem} is $\mu$-strongly convex and $L$-smooth and let  Assumptions \ref{as:Markov_chain}, \ref{as:unbiased}, \ref{as:w_0}, \ref{as:rho_diff} hold. Let problem \eqref{main_problem} be solved by Algorithm \ref{alg:2}. Then for
    $$
    \gamma = \frac{1}{L}, \quad \eta = \frac{\sqrt{L} - \sqrt{\mu}}{\sqrt{L} + \sqrt{\mu}}, \quad T = \mathcal{\tilde O} \left( \taumix \left[ \sqrt{\chi} + \frac{\rho^2}{\lambda_{\min}^2}\right] \log \frac{1}{\varepsilon}\right),
    $$
	it holds that to achieve $\varepsilon$-approximate solution (in terms of $\EE[ f(x) - f(x^*)] \lesssim \varepsilon$) it takes 
	\begin{align*}
    \mathcal{\tilde O} \left( \taumix \left[ \sqrt{\chi} + \frac{\rho^2}{\lambda_{\min}^2}\right] \log \frac{1}{\varepsilon} \cdot \sqrt{\frac{L}{\mu}} \log \frac{1}{\varepsilon} \right) ~~ \text{communications and}
    \\
    \mathcal{O} \left( \sqrt{\frac{L}{\mu}} \log \frac{1}{\varepsilon} \right) ~~ \text{local computations on each node}\,.
	\end{align*}
\end{corollary}
From the point of view of local calculations this result is optimal \cite{nesterov2003introductory}. The situation with communication complexity is more tricky. In the general case the estimate $\mathcal{\tilde O} (\chi \cdot \sqrt{L/\mu})$ from \cite{kovalev2021lower, rogozin2021towards} is optimal \cite{kovalev2021lower}. But our result $\mathcal{\tilde O} ( \taumix [ \sqrt{\chi} + (\rho/\lambda_{\min})^2 ] \cdot \sqrt{L/\mu})$ for the special stochastic Markovian setting can break through the lower bounds from \cite{kovalev2021lower}, e.g., when $\taumix$ and $\rho/\lambda_{\min}$ are quite small. In Section \ref{sec:lower}, we show that deterministic graph changes are more adversarial, and even with the appearance or missing of several edges the lower bounds remain $\tilde \Omega (\chi \cdot \sqrt{L/\mu})$, which means that no acceleration in terms of communications is possible.

\subsection{Proof of Theorem \ref{th:conv_markov_acc}} \label{sec:upper_proofs}

Before proving Theorem \ref{th:conv_markov_acc}, we give the following lemmas.

\begin{lemma}
\label{lem:lem:tech_lemma_403}
For $x^k, x^k_g, x^k_f$ from Algorithm \ref{alg:AGD_ASGD} it holds that $\sum_{j=1}^n x^k_{(j)} = \sum_{j=1}^n (x^k_f)_{(j)} = \sum_{j=1}^n (x^k_g)_{(j)} = \sum_{j=1}^n x^0_{(j)}$. 
\end{lemma}
\begin{proof}
Let us prove by induction. For $x^0, x^0_g, x^0_f$ the statement of Lemma follows from the initialization of $x^0_f = x^0$ and line \ref{line_acc_1}. Suppose that $\sum_{j=1}^n x^k_{(j)} = \sum_{j=1}^n (x^k_f)_{(j)} = \sum_{j=1}^n (x^k_g)_{(j)} = \sum_{j=1}^n x^0_{(j)}$. Let us prove that this is also valid for $x^{k+1}, x^{k+1}_g, x^{k+1}_f$. Using the definition of the gossip matrix, we get $\mathbf{1} \in \ker W^T (\mathcal{G}_{T^{k} + i})$. It means that for $y = W (\mathcal{G}_{T^{k} + i}) x^{k}_g$, we have $\sum_{j=1}^n y_{(j)} = \mathbf{1}^T y = \mathbf{1}^T W (\mathcal{G}_{T^{k} + i}) x^{k}_g = 0$. This fact guarantees that $\sum_{j=1}^n (x^{k+1}_f)_{(j)} = \sum_{i=1}^n (x^k_g)_{(j)}$. The fact $\sum_{j=1}^n x^{k+1}_{(j)} = \sum_{j=1}^n (x^{k+1}_g)_{(j)} = \sum_{i=1}^n x^0_{(j)}$ follows from lines \ref{line_acc_1} and \ref{line_acc_3}.
\end{proof}

\begin{lemma}
\label{lem:lem:tech_lemma_404}
For any $x, y \in \R^n$ such that $\sum_{j=1}^n x_{(j)} = \sum_{j=1}^n y_{(j)}$, it holds  
    \begin{align*}
    r(x) \leq r(y) - \langle \nabla r(x), y - x \rangle - \frac{\lambda_{\min}}{2} \| x - y\|^2.
    \end{align*}
\end{lemma}
\begin{proof}
    If $x = y$, the statement of Lemma follows automatically. In the further course of the proof, we assume that $x \neq y$. 
    
Let us prove by contradiction that $ (x-y) \notin \ker W_0$. If $ (x-y) \in \ker \tilde W$, then $x_{(1)} - y_{(1)} = \ldots = x_{(j)} - y_{(j)} = \ldots = x_{(n)} - y_{(n)}$. From the condition of Lemma it is known that $\sum_{j=1}^n x_{(j)} = \sum_{j=1}^n y_{(j)}$, hence we have that $\sum_{j=1}^n [x_{(j)} - y_{(j)}] =  n [x_{(1)} - y_{(1)}] = 0$ and $x_{(j)} - y_{(j)} = 0$ for all $j \in [n]$. We come to a contradiction, since $x \neq y$.

Finally, we have that $ (x-y) \notin \ker \tilde W$. For such $x$ and $y$, the function $r([x-y]) = \| \sqrt{\tilde W} [x-y]\|^2$ is $\lambda_{\min}$-strongly convex. This completes the proof.
\end{proof}

Also to prove Theorem \ref{th:conv_markov_acc}, we need Lemmas 4, 5 and 6 from \cite{beznosikov2023first}. 

\begin{lemma}[Lemma 4 from \cite{beznosikov2023first}]
	\label{lem:lem:tech_lemma_0}
	Let Assumptions \ref{as:rho_diff}, \ref{as:Markov_chain}, \ref{as:unbiased} hold. Then for the gradient estimates $g^k$ from Algorithm \ref{alg:AGD_ASGD} it holds that $\EE_k[g^k] = \EE_k[g^{k}_{\lfloor \log_2 \batchbound \rfloor}]$. Moreover, 
	\begin{align*}
		&\EE_k[\| \nabla r(x^k_g) - g^k\|^2] \leq 102\left(\taumix B^{-1}\log_2 \batchbound + \taumix^2 B^{-2}\right) \rho^2 \| x^k_g - x^*\|^2\,, \\
		&\| \nabla r(x^k_g) - \EE_{k}[g^k]\|^2 \leq 86\taumix^2 \batchbound^{-2}B^{-2} \rho^2 \| x^k_g - x^*\|^2\,. \nonumber 
	\end{align*}
\end{lemma}

\begin{lemma}[Lemma 5 from \cite{beznosikov2023first}]
	\label{lem:tech_lemma_1}
	For the iterates of Algorithm \ref{alg:AGD_ASGD} with $\theta = (p \eta^{-1} - 1) / (\beta p \eta^{-1} - 1)$, $\theta > 0$, $\eta \geq 1$, it holds that
	\begin{align*}
		\EE_k[\|x^{k+1} - x^*\|^2]
		\leq&
		(1 + \alpha \gamma \eta)( 1 - \beta) \| x^k - x^*\|^2 + (1 + \alpha \gamma \eta) \beta\|x^k_g - x^*\|^2 
		\notag\\
		&
		+ (1 + \alpha \gamma \eta) (\beta^2 - \beta )\|x^k - x^k_g\|^2
		+ p^2 \eta^2 \gamma^2 \EE_{k}[\| g^k \|^2] 
		\notag\\
		&
		- 2 \eta^2 \gamma \langle \nabla r(x^k_g), x^k_g + \left(\frac{p}{\eta} - 1 \right)  x^k_f - \frac{p}{\eta} x^*\rangle 
		\notag\\
		&
		+ \frac{p \eta \gamma}{\alpha} \|\EE_k[g^k] - \nabla r(x^k_g)\|^2\,,
	\end{align*}
	where $\alpha > 0$ is any positive constant.
\end{lemma}

To use the following lemma, we proved Lemmas \ref{lem:lem:tech_lemma_403} and \ref{lem:lem:tech_lemma_404}. 
\begin{lemma}[Lemma 6 from \cite{beznosikov2023first}]
	\label{lem:tech_lemma_2}
	Let problem \eqref{eq:consensus_problem}  be solved by Algorithm \ref{alg:AGD_ASGD}. Then for any $u \in \R^n$ such that $\sum_{i=1}^n u_{(i)} = \sum_{i=1}^n x^0_{(i)}$, we get
	\begin{align*}
		\EE_k[r(x^{k+1}_f)]
		\leq&
		r(u) - \langle \nabla r(x^k_g), u - x^k_g \rangle - \frac{\lambda_{\min}}{2} \| u - x^k_g\|^2  - \frac{\gamma}{2} \|\nabla r(x^k_g)\|^2 
		\\
		&
		+ \frac{\gamma}{2} \|\EE_k[g^k] - \nabla r(x^k_g) \|^2 + \frac{\lambda_{\max} \gamma^2 }{2}\EE_k[\| g^k\|^2].
	\end{align*}
\end{lemma}

\begin{proof}[Proof of Theorem \ref{th:conv_markov_acc}]

	With Lemmas \ref{lem:lem:tech_lemma_403} and \ref{lem:lem:tech_lemma_404}, one can use Lemma \ref{lem:lem:tech_lemma_404}, Lemma \ref{lem:lem:tech_lemma_404} \ref{lem:tech_lemma_1} with $u =  x^*$, $u = x^k_f$ and get
	\begin{align*}
		\EE_k[r(x^{k+1}_f)]
		\leq&
		r(x^*) - \langle \nabla r(x^k_g),  x^* - x^k_g \rangle - \frac{\lambda_{\min}}{2} \|  x^* - x^k_g\|^2  - \frac{p \gamma}{2} \|\nabla r(x^k_g)\|^2 \\
		&+ \frac{p \gamma}{2} \|\EE_k[g^k] - \nabla r(x^k_g) \|^2 + \frac{\lambda_{\max} p^2 \gamma^2 }{2}\EE_k[\| g^k\|^2],
	\end{align*}
	\begin{align*}
		\EE_k[r(x^{k+1}_f)]
		\leq&
		r(x^k_f) - \langle \nabla r(x^k_g), x^k_f - x^k_g \rangle - \frac{\lambda_{\min}}{2} \| x^k_f - x^k_g\|^2  - \frac{p \gamma}{2} \|\nabla r(x^k_g)\|^2 
		\\
		&+ \frac{p \gamma}{2} \|\EE_k[g^k] - \nabla r (x^k_g) \|^2 + \frac{\lambda_{\max} p^2 \gamma^2 }{2}\EE_k[\| g^k\|^2].
	\end{align*}
	Summing the first inequality with coefficient $2 p\gamma \eta  $, the second with coefficient $2 \gamma \eta (\eta - p) $ and the estimate from Lemma \ref{lem:tech_lemma_1}, we obtain
	\begin{align*}
		\EE_k\big[\|x^{k+1} &-  x^*\|^2 + 2 \gamma \eta^2 r(x^{k+1}_f)\big]
		\\
		\notag \leq&
		(1 + \alpha \gamma \eta)( 1 - \beta) \| x^k -  x^*\|^2 + (1 + \alpha \gamma \eta) \beta\|x^k_g -  x^*\|^2 
		\\
		&+ (1 + \alpha \gamma \eta) (\beta^2 - \beta )\|x^k - x^k_g\|^2 
		\\
		&
		- 2 \eta^2 \gamma \langle \nabla r(x^k_g), x^k_g + \left(\frac{p}{\eta} - 1\right)  x^k_f - \frac{p}{\eta}  x^*\rangle 
		\\
		&+ p^2 \eta^2 \gamma^2 \EE_{k}[\| g^k \|^2] + \frac{p \eta \gamma}{\alpha} \|\EE_k[g^k] - \nabla r(x^k_g)\|^2
		\\
		& + 2 p \gamma \eta \Big (r( x^*) - \langle \nabla r(x^k_g),  x^* - x^k_g \rangle - \frac{\lambda_{\min}}{2} \| x^* - x^k_g\|^2  - \frac{p\gamma}{2} \|\nabla r(x^k_g)\|^2 
		\\
		&+ \frac{p\gamma}{2} \|\EE_k[g^k] - \nabla r(x^k_g) \|^2 + \frac{\lambda_{\max} p^2 \gamma^2 }{2}\EE_k[\| g^k\|^2]\Big)
		\\
		& + 2 \gamma \eta (\eta - p) \Big ( r(x^k_f) - \langle \nabla r(x^k_g), x^k_f - x^k_g \rangle - \frac{\lambda_{\min}}{2} \| x^k_f - x^k_g\|^2  - \frac{p \gamma}{2} \|\nabla r(x^k_g)\|^2 
		\\
		&+ \frac{p \gamma}{2} \|\EE_k[g^k] - \nabla r(x^k_g) \|^2 + \frac{\lambda_{\max} p^2\gamma^2 }{2}\EE_k[\| g^k\|^2]  
		\Big)
		\\
		\notag =&
		(1 + \alpha \gamma \eta)( 1 - \beta) \| x^k -  x^*\|^2 + 2 \gamma \eta \left( \eta - p\right) r(x^{k}_f) - 2 p \gamma \eta r( x^*)
		\\
		&\notag
		+ \left((1 + \alpha \gamma \eta) \beta - p\gamma \eta \lambda_{\min}\right)\|x^k_g -  x^*\|^2
		\\
		&\notag
		+ (1 + \alpha \gamma \eta) (\beta^2 - \beta )\|x^k - x^k_g\|^2 
		- p \gamma^2 \eta^2 \|\nabla r(x^k_g)\|^2 
		\\
		&\notag
		+ \left( \frac{p \eta \gamma}{\alpha} + p \gamma^2 \eta^2 \right) \|\EE_k[g^k] - \nabla r(x^k_g) \|^2 + \left( p^2 \eta^2 \gamma^2 + p^2 \gamma^3 \eta^2 \lambda_{\max} \right) \EE_k[\| g^k\|^2]
		\\
		\notag \leq&
		(1 + \alpha \gamma \eta)( 1 - \beta) \| x^k -  x^*\|^2 + 2 \gamma \eta \left( \eta - p\right)  r(x^{k}_f) - 2 p \gamma \eta r( x^*)
		\\
		&\notag
		+ \left((1 + \alpha \gamma \eta) \beta - p\gamma \eta \lambda_{\min}\right)\|x^k_g -  x^*\|^2
		\\
		&\notag
		+ (1 + \alpha \gamma \eta) (\beta^2 - \beta )\|x^k - x^k_g\|^2 
		- p \gamma^2 \eta^2 \|\nabla r(x^k_g)\|^2 
		\\
		&\notag
		+ p \eta \gamma \left( \frac{1}{\alpha} + \gamma \eta \right) \|\EE_k[g^k] - \nabla r(x^k_g) \|^2 
		\\
		&\notag
		+ 2 p^2 \eta^2 \gamma^2 \left(  1 +  \gamma \lambda_{\max} \right) \EE_k[\| g^k - \nabla r(x^k_g)\|^2] 
		\\
		&\notag
		+ 2 p^2 \eta^2 \gamma^2 \left(  1 +  \gamma \lambda_{\max} \right) \EE_k[\| \nabla r(x^k_g) \|^2]\,.
	\end{align*}
	In the last step we also used Cauchy Schwartz inequality. The choice of $\alpha = \frac{\beta}{2\eta \gamma}$ gives $(1 + \alpha \eta \gamma) (1 - \beta) = \left(1 + \frac{\beta}{2}\right) \left( 1 - \beta\right) \leq \left( 1 - \frac{\beta}{2}\right)$ and then
	\begin{align*}
		\EE_k[\|x^{k+1} &-  x^*\|^2 + 2 \gamma \eta^2 r(x^{k+1}_f)]
		\\
		\notag \leq&
		\left( 1 - \frac{\beta}{2} \right) \| x^k -  x^*\|^2 + 2 \gamma \eta \left( \eta - p\right)  r(x^{k}_f) - 2 p \gamma \eta r( x^*)
		\\
		&\notag
		+ \left( \left(1 + \frac{\beta}{2} \right) \beta - p\gamma \eta \lambda_{\min}\right)\|x^k_g -  x^*\|^2
		+ \left(1 + \frac{\beta}{2} \right) (\beta^2 - \beta )\|x^k - x^k_g\|^2 
		\\
		&\notag
		+ p \eta^2 \gamma^2 \left( 1 + \frac{2}{\beta} \right) \|\EE_k[g^k] - \nabla r(x^k_g) \|^2
		\\
		&\notag
		+ 2 p^2 \eta^2 \gamma^2 \left(  1 +  \gamma \lambda_{\max} \right) \EE_k[\| g^k - \nabla r(x^k_g)\|^2] 
		\\
		&\notag
		- p \eta^2 \gamma^2 \left(1 -  2p\left(  1 +  \gamma \lambda_{\max} \right) \right) \EE_k[\| \nabla r(x^k_g) \|^2]\,.
	\end{align*}
	Subtracting $2\gamma \eta^2 r( x^*)$ from both sides, we get
	\begin{align*}
		\EE_k[\|x^{k+1} &-  x^*\|^2 + 2 \gamma \eta^2 (r(x^{k+1}_f) - r( x^*))]
		\\
		\notag \leq&
		\left( 1 - \frac{\beta}{2} \right) \| x^k -  x^*\|^2 + 2 \gamma \eta^2 \left( 1 - \frac{p}{\eta} \right)  (r(x^{k}_f) -  r( x^*))
		\\
		&\notag
		+ \left( \left(1 + \frac{\beta}{2} \right) \beta - p\gamma \eta \lambda_{\min}\right)\|x^k_g -  x^*\|^2
		+ \left(1 + \frac{\beta}{2} \right) (\beta^2 - \beta )\|x^k - x^k_g\|^2 
		\\
		&\notag
		+ p \eta^2 \gamma^2 \left( 1 + \frac{2}{\beta} \right) \|\EE_k[g^k] - \nabla r(x^k_g) \|^2 \\
		&\notag
		+ 2 p^2 \eta^2 \gamma^2 \left(  1 +  \gamma \lambda_{\max} \right) \EE_k[\| g^k - \nabla r(x^k_g)\|^2] 
		\\
		&\notag
		- p \eta^2 \gamma^2 \left(1 -  2p\left(  1 +  \gamma \lambda_{\max} \right) \right) \EE_k[\| \nabla r(x^k_g) \|^2]\,.
	\end{align*}
	Applying Lemma \ref{lem:lem:tech_lemma_0}, one can obtain
	\begin{align*}
		\EE_k[\|x^{k+1} &-  x^*\|^2 + 2 \gamma \eta^2 (r(x^{k+1}_f) - r( x^*))]
		\\
		\notag \leq&
		\left( 1 - \frac{\beta}{2}\right) \| x^k -  x^*\|^2 + 2 \gamma \eta^2 \left( 1 - \frac{p}{\eta} \right)  (r(x^{k}_f) -  r( x^*))
		\\
		&\notag
		+ \left( \left(1 + \frac{\beta}{2} \right) \beta - p\gamma \eta \lambda_{\min}\right)\|x^k_g -  x^*\|^2
		+ \left(1 + \frac{\beta}{2} \right) (\beta^2 - \beta )\|x^k - x^k_g\|^2 
		\\
		&\notag
		+ p \eta^2 \gamma^2 \left( 1 + \frac{2}{\beta} \right) \cdot 86\taumix^2 \batchbound^{-2}B^{-2} \rho^2 \| x^k_g -  x^*\|^2 
		\\
		&\notag
		+ 2 p^2 \eta^2 \gamma^2 \left(  1 +  \gamma \lambda_{\max} \right) \cdot 102 \left(\taumix B^{-1}\log_2 \batchbound + \taumix^2 B^{-2}\right) \rho^2 \| x^k_g -  x^*\|^2
		\\
		&\notag
		- p \gamma^2 \eta^2 ( 1 - 2 p (  1 +  \gamma \lambda_{\max}))  \| \nabla r(x^k_g) \|^2\,.
	\end{align*}
	With $M \geq \sqrt{\frac{1 + 2/\beta}{p}}$, we have
	\begin{align*}
		\EE_k[\|x^{k+1} &-  x^*\|^2 + 2 \gamma \eta^2 (r(x^{k+1}_f) - r( x^*))]
		\\
		\notag \leq&
		\left( 1 - \frac{\beta}{2}\right) \| x^k -  x^*\|^2 + 2 \gamma \eta^2 \left( 1 - \frac{p}{\eta} \right)  (r(x^{k}_f) -  r( x^*))
		\\
		&\notag
		+ \left(\left(1 + \frac{\beta}{2} \right) \beta - p\gamma \eta \lambda_{\min}\right)\|x^k_g -  x^*\|^2
		+ \left(1 + \frac{\beta}{2} \right) (\beta^2 - \beta )\|x^k - x^k_g\|^2 
		\\
		&\notag
		+ 300 p^2 \eta^2 \gamma^2 \rho^2  \left(  1 +  \gamma \lambda_{\max} \right) \left(\taumix B^{-1}\log_2 \batchbound + \taumix^2 B^{-2}\right) \| x^k_g -  x^*\|^2
		\\
		&\notag
		- p \gamma^2 \eta^2 ( 1 - 2 p (  1 +  \gamma \lambda_{\max}))  \| \nabla r(x^k_g) \|^2\,.
	\end{align*}
	With $\gamma \leq \tfrac{3}{4\lambda_{\max}}$, using that $p = \tfrac{1}{4}$, $\beta  = \sqrt{\frac{4p^2 \lambda_{\min} \gamma}{3}}$, and $p \lambda_{\min} \gamma \eta = 3\beta$, one can obtain
	\begin{eqnarray*}
		&\beta  = \sqrt{\frac{4p^2 \lambda_{\min} \gamma}{3}} \leq \sqrt{ \frac{p^2 \lambda_{\min}}{\lambda_{\max}}} \leq 1, \\
		& 1 + \gamma \lambda_{\max} \leq 2, \quad -( 1 -  2p(  1 +  \gamma \lambda_{\max})) \leq 0,\\
		&\left( \left(1 + \frac{\beta}{2} \right) \beta - p\gamma \eta \lambda_{\min}\right) = \left( \beta + \frac{\beta^2}{2} - p \lambda_{\min} \gamma \eta \right) \leq \left(\frac{3 \beta}{2} - p \lambda_{\min} \gamma \eta \right) \leq - \frac{p \lambda_{\min} \gamma \eta}{2}.
	\end{eqnarray*}
	and, therefore,
	\begin{align*}
		\EE_k[\|x^{k+1} &-  x^*\|^2 + 2 \gamma \eta^2 (r(x^{k+1}_f) - r( x^*))]
		\\
		\notag \leq&
		\left( 1 - \frac{\beta}{2} \right) \| x^k -  x^*\|^2 + 2 \gamma \eta^2 \left( 1 - \frac{p}{\eta} \right)  (r(x^{k}_f) -  r( x^*))
		\\
		&\notag
		+ \left( - \frac{p\gamma \eta \lambda_{\min}}{2}  + 300 p^2 \eta^2 \gamma^2 \rho^2  \left(\taumix B^{-1}\log_2 \batchbound + \taumix^2 B^{-2}\right)\right)\|x^k_g -  x^*\|^2\,.
	\end{align*}
	Since $p=\frac{1}{4}$, $\eta = \sqrt{\frac{12}{\lambda_{\min} \gamma}}$, $\gamma \leq \tfrac{\lambda_{\min}^3}{[1800 \rho^2 \left(\taumix b^{-1}  + \taumix^2 b^{-2}\right)]^{2}}$ and $B = \lceil b \log_2 \batchbound \rceil$, we get
	\begin{eqnarray*}
		&\gamma 
		\leq 
		\frac{\lambda_{\min}^3}{[1800 \rho^2 \left(\taumix b^{-1}  + \taumix^2 b^{-2}\right)]^{2}} 
		\leq \frac{\lambda_{\min}^3}{[1800 \rho^2 \left(\taumix B^{-1} \log_2 \batchbound + \taumix^2 B^{-2}\right)]^{2}} 
		,\\
		&\left( - \frac{p\gamma \eta \lambda_{\min}}{2}  + 300 p^2 \eta^2 \gamma^2 \rho^2  \left(\taumix B^{-1}\log_2 \batchbound + \taumix^2 B^{-2}\right)\right)
		\\
		&\hspace{0.8cm}=
		\frac{p\gamma \eta }{2} \left( - \lambda_{\min}  + 150 \eta \gamma \rho^2  \left(\taumix B^{-1}\log_2 \batchbound + \taumix^2 B^{-2}\right)\right)
		\\
		&\hspace{2.4cm}=
		\frac{p\gamma \eta }{2} \left(-\lambda_{\min}  + \sqrt{\frac{12 \cdot 150^2 \gamma}{\lambda_{\min}}} \cdot \rho^2 \left(\taumix B^{-1}\log_2 \batchbound + \taumix^2 B^{-2}\right)\right)
		\leq 0,
	\end{eqnarray*}
	and, then,
	\begin{align*}
		\EE_k[\|x^{k+1} -  x^*\|^2 &+ 2 \gamma \eta^2 (r(x^{k+1}_f) - r( x^*))]
		\\
		\notag \leq&
		\left( 1 - \frac{\beta}{2} \right) \| x^k -  x^*\|^2 + 2 \gamma \eta^2 \left( 1 - \frac{p}{\eta} \right)  (r(x^{k}_f) -  r( x^*))\,.
	\end{align*}
	Using that $\eta = \sqrt{\frac{12}{\lambda_{\min} \gamma}} $ and $\frac{\beta}{2} = \frac{p}{\eta}$, we have
	\begin{align*}
		\EE_k \bigg[\|x^{k+1} -  x^*\|^2 &+ \frac{24}{\lambda_{\min}} (r(x^{k+1}_f) - r( x^*)) \bigg]
		\\
		\notag \leq&
		\left( 1 - \frac{\beta}{2} \right) \left[\| x^k -  x^*\|^2 + \frac{24}{\lambda_{\min}}  (r(x^{k}_f) -  r( x^*)) \right]\,.
	\end{align*}
	Substituting of $\beta = \sqrt{\frac{4p^2 \lambda_{\min} \gamma}{3}}$, we have
    \begin{align*}
		\EE_k \bigg[\|x^{k+1} -  x^*\|^2 &+ \frac{24}{\lambda_{\min}} (r(x^{k+1}_f) - r( x^*)) \bigg]
		\\
		\notag \leq&
		\left( 1 - \sqrt{\frac{p^2 \lambda_{\min} \gamma}{3}} \right) \left[\| x^k -  x^*\|^2 + \frac{24}{\lambda_{\min}}  (r(x^{k}_f) -  r( x^*)) \right]\,.
	\end{align*}
	Taking the full expectation and running the recursion finish the proof.
\end{proof}

\section{Lower bounds} \label{sec:lower}

\subsection{First-order decentralized algorithms}
Similar to the works \cite{scaman2017optimal, metelev2023consensus, kovalev2021lower} we will impose some conditions on the optimization algorithm. We will call the class of algorithms satisfying these conditions \textit{ first-order decentralized algorithms}.
Each algorithm in this class has two types of iterations: communicational and local. In the communicational iteration, the nodes communicate with each other, while in the local iteration, they perform computations on their local memory. For each time step $k\in\N$ we will call $\cH_i(k)$ the local memory of the node $i$. Also for each time step $k\in\N$, denote the last preceding communication time by $q(k)$.
1. If nodes perform a local computation at step $k$, local information is updated as
    \begin{align*}
    \cH_i(k+1)\subseteq \spn\cbraces{\braces{x, \nabla f_i(x), \nabla f_i^*(x):~ x\in\cH_i(k)}}
    \end{align*}
    for all $i = 1, \ldots, n$. Here $f^*(y)=\sup_x\{y^T x-f(x)\}$ is the Fenchel's dual function.
    
2. If the nodes perform a communication round at time step $k$, local information is updated as
    \begin{align*}
    \cH_i(k+1)\subseteq \spn\cbraces{\bigcup_{j\in\cN_i^{q(k)}\cup\{i\}} \cH_j(k)}
    \end{align*}
    for all $i = 1, \ldots, n$. Here $\cN_i^{q(k)}$ is a set of neighbors of node $i$ at time step $q(k)$, i.e. at the time of last communication.


\subsection{Overview of main result}

In order to establish a new lower bound, we will employ a slightly different concept than that found in previous works \cite{scaman2017optimal, metelev2023consensus, kovalev2021lower}. The fundamental idea behind such lower bound approaches is to construct a counterexample of a time-varying network, in which information flows slowly from one large cluster to another, while maintaining a modest characteristic number for the network. In previous studies \cite{scaman2017optimal,kovalev2021lower}, this was achieved by utilizing classical (unweighted) graph Laplacians corresponding to the network. In our novel approach, we employ a weighted Laplacian to construct a counterexample.

To estimate the characteristic number of a graph, it is necessary to evaluate both the largest and smallest nonzero eigenvalues of the Laplacian. Although maximal eigenvalue can be easily estimated, difficulties arise in determining the minimal positive eigenvalue. The literature contains numerous typical graphs (such as paths, stars, and complete binary trees) for which lower nonzero eigenvalues of Laplacians have been calculated, and these can be employed for counterexamples. Our approach, which is based on weighted Laplacians, allows to utilize previously inaccessible topologies.

Now we formulate our result on lower bounds as follows. The proof will be provided in the forthcoming sections.
\begin{theorem}\label{thm:lower_bounds}
    For any $\chi\geq 56,~ L>16\mu>0$ and any first-order optimization method $\cM$ there exists a set of $L$-smooth and $\mu$-strongly convex functions $\braces{f_i}_{i=1}^m$, a sequence of graphs $\braces{\cG_k = (\cV, \cE_k)}_{k=0}^\infty$ and a sequence of corresponding gossip matrices $\braces{W(\cG_k)}_{k=0}^\infty$ such that for each $k = 0, 1, \ldots$ it holds $\chi(W(\cG_k)) = \chi$ and
    \begin{equation*}
        \norm{x_k-x_*}^2 \ge \left(1-4\sqrt{\frac{\mu}{L}}\right)^{\frac{72k}{7\chi}+2}\norm{x_0-x_*}^2.
    \end{equation*}
\end{theorem}
\begin{corollary}
    For any $\chi\geq 56,~ L>16\mu>0$ and any first-order optimization method $\cM$ there exists a set of $L$-smooth and $\mu$-strongly convex functions $\braces{f_i}_{i=1}^m$, a sequence of graphs $\braces{\cG_k = (\cV, \cE_k)}_{k=0}^\infty$ and a sequence of corresponding gossip matrices $\braces{W(\cG_k)}_{k=0}^\infty$ such that for each $k = 0, 1, \ldots$ it holds $\chi(W(\cG_k)) = \chi$ and method $\cM$ requires at least $\Omega(\chi\sqrt{L/\mu}\log(1/\eps))$ communication rounds.
\end{corollary}

\subsection{Auxiliary lemmas for weighted Laplacians}

The weighted Laplacian enables us to adjust the edge weights in such a manner that the smallest nonzero eigenvalue can be easily estimated while simultaneously maintaining control over the largest eigenvalue. We upper bound the largest eigenvalue through the following lemma.
\begin{lemma} \label{lambda_1_upper_bound}
	Let $\cG_A=(\cV, \cE, A)$ be a weighted graph. Let $d_{\max}(\cG_A)=\max_{i\in \cV}{\sum_{(i, j)\in\cE}a_{ij}}$ denote the maximum vertex degree of $\cG_A$. Then $\lambda_{\max}(L(\cG_A)) \leq 2d_{\max}(\cG_A)$.
\end{lemma}
\begin{proof}
	This is a classical result, let us give a proof of it. Let $d_i = \sum_{(i, k)\in\cE} a_{ik}$.

	First, we observe that $\sum_{j=0}^{n} [L(\cG_A)]_{ij} = 0$ for each $i$. This implies that the row sum of the $i$-th row coincides with the diagonal element $d_i$:
	
	$$
	\sum_{j \neq i} |[L(\cG_A)]_{ij}| = \sum_{(i, j)\in\cE} a_{ij} = d_i.
	$$
	
	Let $\ol B(a, R)$ denote a closed disc in complex plane with center at $a$ and radius $R$. By the Gershgorin circle theorem, we have
	$$
	\lambda_1 \in \bigcup_{i=1}^{n} \overline{B}(d_i, d_i)=\overline{B}(d_{max}(\cG_A), d_{max}(\cG_A)).
	$$

	As a result, $\lambda_{\max}(\cG_A) \leq 2d_{\max}(\cG_A)$.
\end{proof}

\begin{lemma}\label{chi_upper}
    For any unweighted graph $\cG=(\cV,\cE)$ there exists a weighted graph $\cG_A = (\cV, \cE, A)$ such that $\chi(L(\cG_A)) \le 2Dn$, where $n = |\cV|$ and $D$ is the diameter of $\cG$.
\end{lemma}
\begin{proof}


For each edge $(i, j)\in\cE$, let $s_{ij}$ denote an arbitrary shortest path from $i$ to $j$ (in the unweighted graph $\cG$). Let us build a weighted graph $\cG_A = (\cV, \cE, A)$ in the following way. Assign each of the edges $(i, j)\in\cE$ a weight $a_{ij}$ equal to the number of shortest paths traversing edge $(i, j)$, i.e.
\begin{align*}
	a_{ij} = \left|s_{kl}:~ (i, j)\in s_{kl}\right|,
\end{align*}
and set $a_{ij} = 0$ if $(i, j)\notin\cE$. Note that the Laplacian of $\cG_A$ has the form
\begin{equation*}
	L(\cG_A) = \sum_{(i, j) \in E} a_{ij} \ell_{ij}.
\end{equation*}
For a fixed shortest path $s_{ij}$ of length $d$ in the graph, let $s_{ij} = \{i=v_0, \ldots, v_d=j\}$, and let $S_{ij}$ denote the sum of mini-Laplacians corresponding to edges in the path $s_{ij}$, i.e.,
\begin{align*}
S_{ij} = \ell_{i v_1} + \ell_{v_1 v_2} + \ldots + \ell_{v_{d-1} j}.
\end{align*}
Note that the Laplacian of $\cG_A$ can be also written as
\begin{align*}
L(\cG_A) = \sum_{(i, j)\in\cE} S_{ij}.
\end{align*}

To estimate $\lambda_{n-1}(L(\cG_A))$, we use the theorem stating that if $A, B$ are symmetric matrices and $A \succeq B$, then $\lambda_i(A) \geq \lambda_i(B)$ for all $i \in {1, \dots, n}$, where $\lambda_i(A)$ and $\lambda_i(B)$ are sorted in a descending order. We will show that for any shortest path $s_{ij}$ with length $d$,

\begin{align*}
S_{ij} = \ell_{i v_1} + \ell_{v_1 v_2} + \dots + \ell_{v_{d-1} j} \succeq \frac{1}{d} \ell_{ij}.
\end{align*}
Let $x \in \mathbb{R}^n$. Then
\begin{align*}
x^T S_{ij} x &= x^T (\ell_{i v_1} + \ell_{v_1 v_2} + \dots + \ell_{v_{d-1} j}) x
= \frac{1}{d} \left(d\sum_{l=0}^{d-1} (x_{v_l} - x_{v_{l+1}})^2\right) \\
&= \frac{1}{d} \left(\left(\sum_{l=0}^{d-1}1^2\right)\sum_{l=0}^{d-1} (x_{v_l} - x_{v_{l+1}})^2\right)
\overset{\circledOne}{\geq} \frac{1}{d} \left(\sum_{l=0}^{d-1} 1\cdot (x_{v_l} - x_{v_{l+1}})\right)^2 \\
&\ge \frac{1}{d} (x_{v_0}-x_{v_d})^2
= \frac{1}{d} x^T \ell_{ij} x,
\end{align*}
where $\circledOne$ holds by Cauchy–Schwarz inequality. Since $x^\top s_{ij} x\geq \frac{1}{d}x^\top \ell_{ij} x$ holds for any $x \in \mathbb{R}^n$, we have
\begin{align*}
S_{ij} \succeq \frac{1}{d} \ell_{ij}.
\end{align*}
Now, we sum these lower bounds for all pairs of distinct vertices $i$ and $j$ in the graph, obtaining
\begin{align*}
L(\cG_A) = \sum_{(i, j)\in\cE} S_{ij} \succeq \frac{1}{D} \sum_{(i, j)\in\cE} \ell_{ij}.
\end{align*}
Thus, our matrix $L(\cG_A)$ is not smaller than the Laplacian corresponding to the complete graph divided by $D$. As it is known, the spectrum of the complete graph consists of the number $n$ with multiplicity $n-1$ and $0$ with multiplicity $1$. Therefore,
\begin{equation}\label{chi_upper_through_d_max}
\lambda_{\min}^+(L(\cG_A)) \geq \frac{n}{D}.
\end{equation}
Next, by Lemma~\ref{lambda_1_upper_bound} we have that $\lambda_1(L(\cG_A))\leq 2d_{\max}(\cG_A)$. To estimate $d_{\max}(\cG_A)$, note that each shortest path from $i$ to $j$ passes through a fixed vertex $v$ at most once, so the adjacent edges to $v$ participate in the sum of mini-Laplacians for the path from $i$ to $j$ at most twice. There are $\frac{n(n-1)}{2}$ pairs of vertices. Therefore, 
$$
d_{\max}(\cG_A) \leq 2\frac{n(n-1)}{2} < n^2.
$$
As a result, we have
\begin{align*}
\frac{\lambda_{\max}(L(\cG_A))}{\lambda_{\min}^+(L(\cG_A))} \leq \frac{2n^2}{n/D}\leq 2nD,
\end{align*}
which completes the proof.
\end{proof}

\subsection{Counterexample graph sequence}

The structure of the further proof will then be identical to the structures in the articles \cite{scaman2017optimal, metelev2023consensus, kovalev2021lower}. Let us describe the structure of the counterexample network.

\begin{definition}
	Consider two star graphs with the number of vertices $a$ and $b$, respectively. Let us add an additional isolated vertex to these two graphs and connect it with edges to the centers of the stars. We denote the resulting graph by $T_{a, b}$.
\end{definition}

\begin{figure}[ht]
    \centering
    \includegraphics[width=0.3\textwidth]{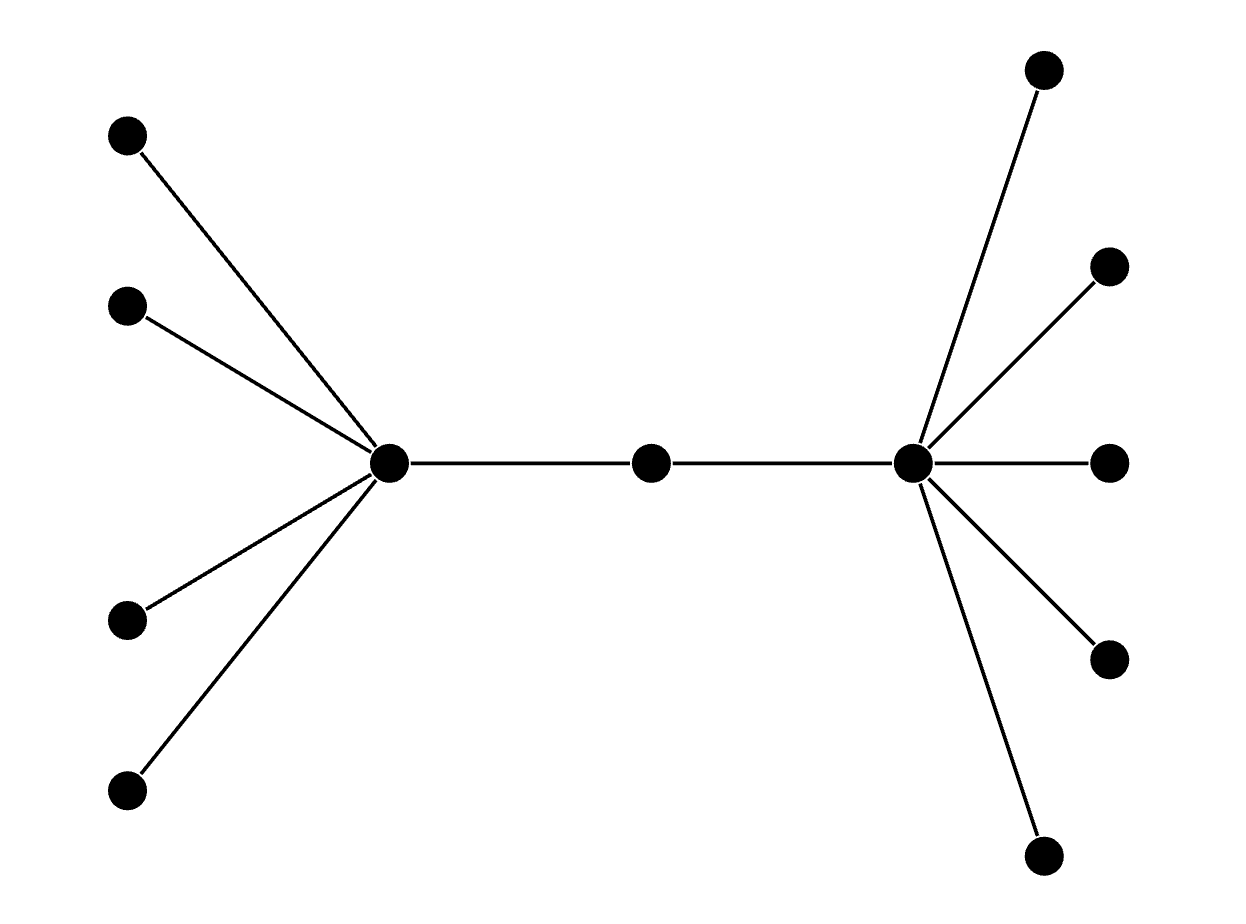}
    \caption{Graph $T_{4, 5}$.}
\end{figure}

Consider the graph $T_{n, n}$, which consists of two "glued-together stars": the left and the right. We will refer to the set of pendant vertices adjacent to the center of the left star as the left partition. The right partition is defined similarly. Let us take the left partition and select $\left[\frac{n}{2}\right]$ vertices from it, denoting them $\cV_1$. In the right partition, also select $\left[\frac{n}{2}\right]$ vertices and denote them $\cV_2$. Next, we will introduce functions on the vertices of this graph. 

\begin{figure}[ht]
    \centering
    \includegraphics[width=0.3\textwidth]{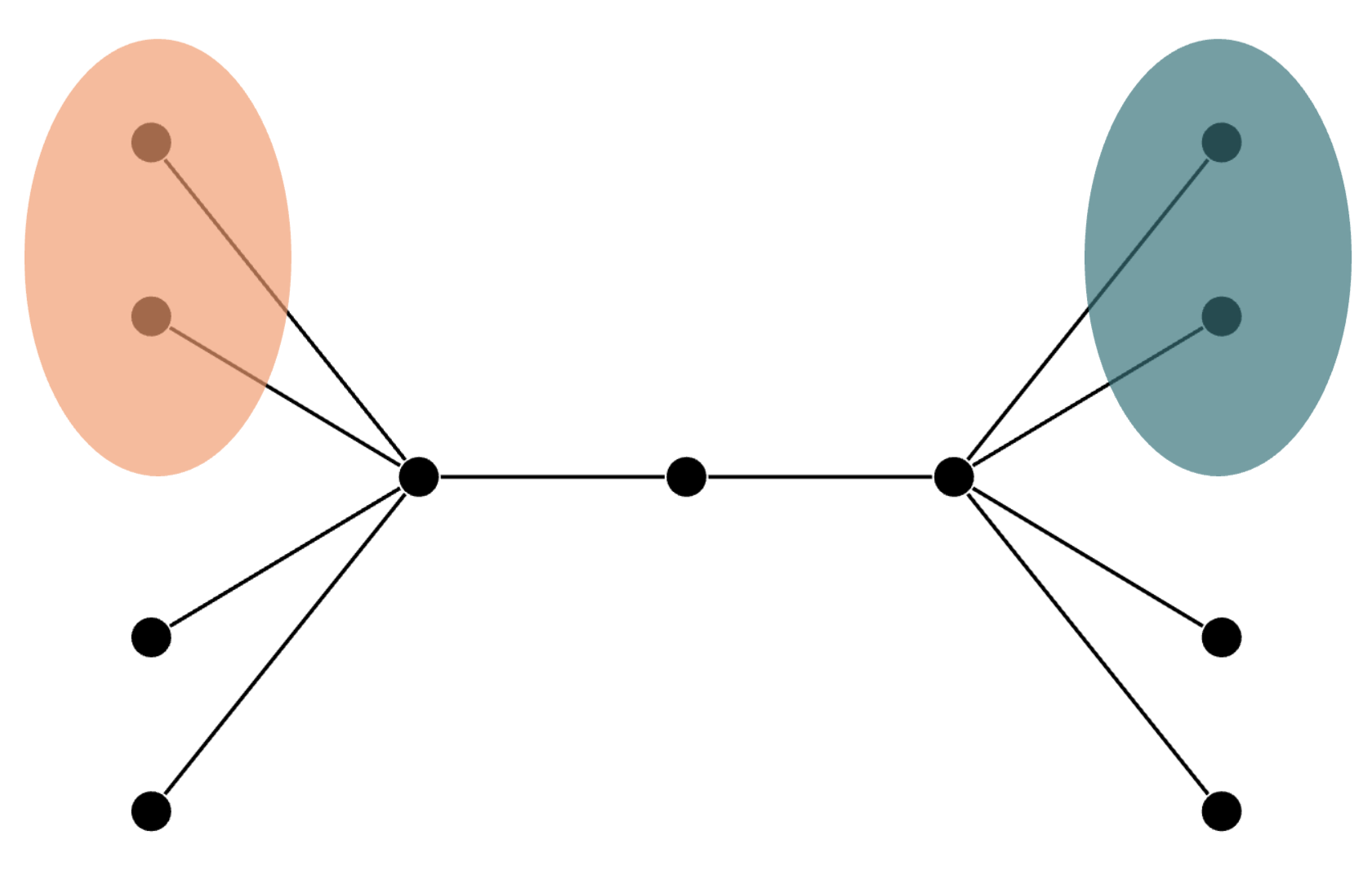}
    \caption{An example of the graph $T_{4, 4}$ with highlighted sets $\mathcal{V}_1$ and $\mathcal{V}_2$. The vertices in the red region belong to the set $\mathcal{V}_1$, while the vertices in the blue region belong to the set $\mathcal{V}_2$.}
\end{figure}

From now on, let us consider a graph $T_{n, n}$ with $n \ge 2$.

Denote the vertex functions $f_v: \ell_2\rightarrow \R$ depending on vertex type:
\begin{equation}\label{poly_func}
f_v(x)=\begin{cases}
 \frac{\mu}{2n}\norm{x}^2+\frac{L-\mu}{4|\cV_2|}\sum_{k=1}^{\infty}(x_{2k-1}-x_{2k})^2, &  v\in \cV_1,\\
 \frac{\mu}{2n}\norm{x}^2+\frac{L-\mu}{4|\cV_1|}\[(x_1-1)^2+\sum_{k=1}^{\infty}(x_{2k}-x_{2k+1})^2\], & v\in \cV_2, \\
 \frac{\mu}{2n}\norm{x}^2, & v\in \cV\setminus\cV_1\setminus\cV_2.
\end{cases}
\end{equation}

From the definition of $\cV_1$ and $\cV_2$ we can deduce that
\begin{equation}
 |\cV_1|=|\cV_2|\ge \frac{n}{4}. \label{werwe}
\end{equation}
Let us estimate the network's global characteristic number using the local one
\begin{equation*}
 \kappa_l = \frac{\frac{L-\mu}{|\cV_1|}+\frac{\mu}{n}}{\frac{\mu}{n}} = \frac{n}{|\cV_1|}(\kappa_g-1)+1 \le 4(\kappa_g-1)+1.
\end{equation*}
 Thus, we have
\begin{equation}\label{kappa_rels}
 \kappa_g \ge \frac{\kappa_l - 1}{4}+1.
\end{equation}

In the following, we will describe the structure of changes in our graph. In total, our graph sequence will be divided into two alternating phases: Phase 1 - "flow of vertices from left to right" and Phase 2 - "flow of vertices from right to left." Let us consider the very first graph in our sequence. We will take the previously described graph $T_{n, n}$ with highlighted sets $\cV_1, \cV_2$ and transfer all unmarked vertices of the right partition to the left partition. The resulting graph will have the form $T_{2n-\left[\frac{n}{2}\right], \left[\frac{n}{2}\right]}$, which will be the first element of our sequence.

Next, let us consider the first phase and define it by induction. We have defined the first graph of the first phase. Suppose that we have the $n$-th element of the first phase and let it have the form $T_{a, b}$, where $a+b=2n$, $a\ge[n/2] + 1$, $b \ge [n/2]$, and we assume that the whole set $\cV_1$ belongs to the left partition and the set $\cV_2$ to the right partition. Let us denote the left center of the star as $l$, and the right center as $r$. Let us denote the vertex connecting $l$ and $r$ by $v$, and denote by $u$ any unmarked vertex of the left partition. We then transfer the vertex $v$ to the right partition and place the vertex $u$ in the position of $v$ (that is, make it the connecting vertex). As a result, we removed the edge $(v, l)$ and added the edge $(u, r)$. As a result, we get a graph of the form $T_{a-1,b+1}$.

\begin{figure}[ht]
    \centering
    \includegraphics[width=0.6\textwidth]{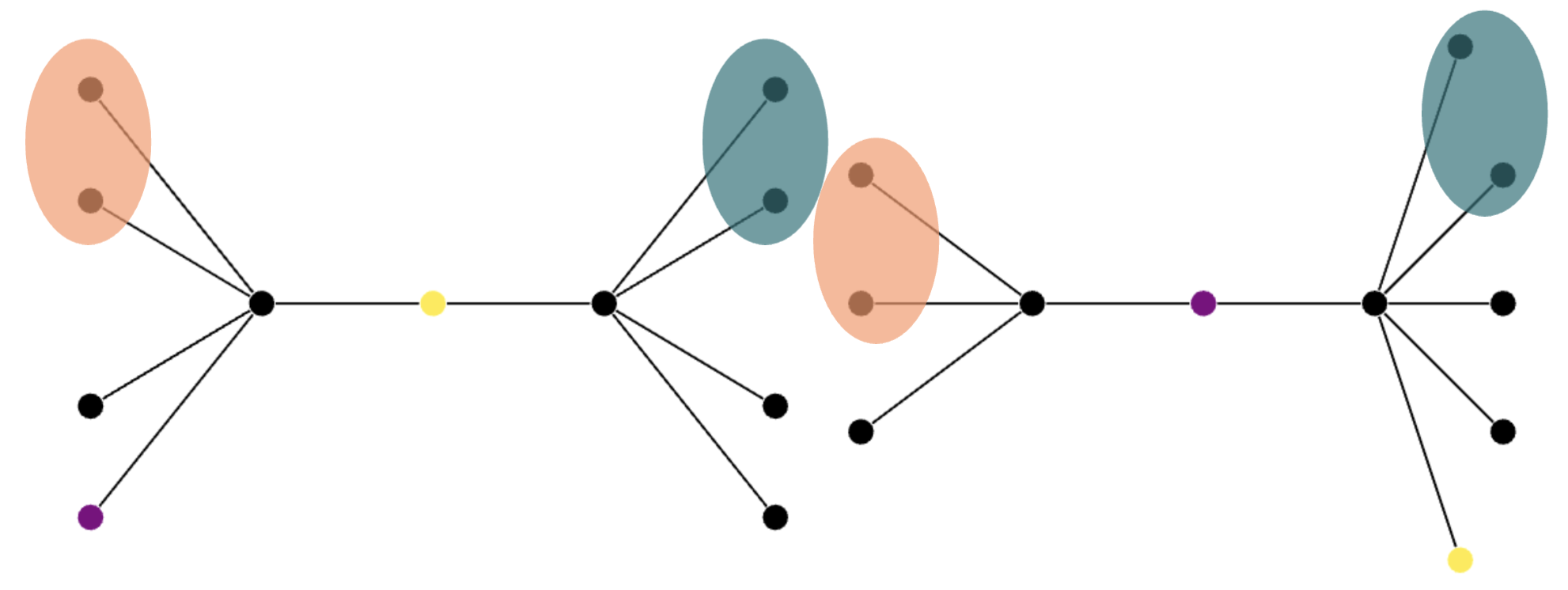}
    \caption{Here is an example of the graph changing scheme. In the left image, the 3rd iteration of Phase 1 is illustrated, and in the right image, the 4th iteration is presented. The red and blue areas respectively consist of the sets $\cV_1$ and $\cV_2$. The vertex $v$ is marked in yellow, while the vertex $u$ is indicated in purple.}
    \label{fig:changing_scheme_example}
\end{figure}

The first phase continues until the graph takes the form $T_{\left[\frac{n}{2}\right], 2n-\left[\frac{n}{2}\right]}$. At this iteration, the first phase ends and the second phase begins (i.e. they intersect at one element). The second phase is defined similarly but symmetrically, with vertices now "flowing" from the right partition to the left partition.

Now, let us consider a model where we assume that the vertices in $\cV_2$ contain some information in their local memory, and they need to transfer this information to the vertices in $\cV_1$. Vertices perform a communication iteration after each graph change and can share information with their neighbors. If we start to change the graph from the first iteration of the first phase, it is not difficult to see that a minimum of $2n - 2\left[\frac{n}{2}\right]$ communication iterations will be required for the information transfer (that is, in the graph at the number $2n - 2\left[\frac{n}{2}\right]$, the vertices $\cV_1$ will not yet possess the information). Similarly, the reasoning for transferring information from vertices $\cV_1$ to $\cV_2$ can be applied during the second phase.

Let us define $t$ as the time it takes to transfer information from one partition to the other one. We have
\begin{equation}\label{flow_lower}
 t=2n-2\left[\frac{n}{2}\right]\ge n.
\end{equation}

\subsection{Proof of Theorem~\ref{thm:lower_bounds}}

The next explanation is taken from \cite{metelev2023consensus}.

Let $x_0=0$ be the initial point for the first-order decentralized algorithm. For every $m\ge 1$, we define $l_m=\min\{k\ge 1|\exists v: \exists x\in \cH_{v}(k): x_m \neq 0\}$ as the first moment when we can get a nonzero element at the $m$-th place at any node.

Considering the functions on vertices from $\cV_1$ and $\cV_2$, we can conclude that functions on vertices from $\cV_1$ can "transfer" (by calculating the gradient) information (nonzero element) from the odd positions ($1, 3, 5,\ldots$) to the next even positions ($2, 4, 6, \ldots$ correspondingly). At the same time, functions on vertices from $\cV_2$ can transfer information from the even positions ($2, 4, 6,\ldots$) to the next odd positions ($3, 5, 7, \ldots$). Therefore, for the network to get a new nonzero element at the next position, a whole phase is required, that is, $t$ communication iterations.

To reach the $m$-th nonzero element, we need to make at least $m$ local steps and $(m - 1)t$ communication steps to transfer information from gradients between $\cV_1$ and $\cV_2$ sets. Therefore, the time $l_m$ at which $m$-th element becomes nonzero can be estimated as
\begin{equation}\label{non_zero_inf}
 l_m \ge (m-1)t+m.
\end{equation}

The solution of the global optimization problem is $x^*_k = \left(\frac{\sqrt{\kappa_g} - 1}{\sqrt{\kappa_g} +1}\right)^k$.

For any $m, k$ such that $l_m > k$
\begin{equation*}
 \norm{x_k-x_*}^2 \ge (x_*)_m^2+(x_*)_{m+1}^2+\ldots=\left(\frac{\sqrt{\kappa_g} - 1}{\sqrt{\kappa_g} +1}\right)^m\norm{x_0-x_*}^2.
\end{equation*}

Using \eqref{non_zero_inf} we can take $m=\ceil{\frac{k}{t+1}}+1$. From \eqref{kappa_rels} we conclude that $\frac{\sqrt{\kappa_g} - 1}{\sqrt{\kappa_g} +1} \ge 1 - \frac{4}{\sqrt{\kappa_l}}$

Therefore using \eqref{kappa_rels} and \eqref{flow_lower} we get
\begin{equation*}
 \norm{x_k-x_*}^2 \ge \left(1-4\sqrt{\frac{\mu}{L}}\right)^{\ceil{k/\(n+1\)}+1}\norm{x_0-x_*}^2.
\end{equation*}

For each graph in our sequence we map a weighted Laplacian from \eqref{chi_upper}, so $\chi\le 8n$.

Rearranging an expression, we get
\begin{equation}\label{arg_lower}
 \norm{x_k-x_*}^2 \ge \left(1-4\sqrt{\frac{\mu}{L}}\right)^{\frac{8k}{\chi}+2}\norm{x_0-x_*}^2.
\end{equation}

Note that our reasoning holds for any value $\chi,L,\mu$ satisfying the following conditions: $\chi = 8(2n+3), n\in\N, n\geq 2$, $L>16\mu>0$. Let us explain why such an estimation is appropriate for any $\chi\geq 56$. Let $\chi\geq 56$, take the closest to it from below $\chi_0=8(2n+3)\geq 56$. Then it is not difficult to see that 
\begin{equation}\label{chi_0_chi_relation}
    \chi_0/\chi\geq \frac{7}{9}.
\end{equation}
Let us take our sequence of graph counterexamples for values $\chi_0, L, \mu$. Take a lower bound for them
\begin{equation}\label{chi_0_bound}
    \norm{x_k-x_*}^2 \ge \left(1-4\sqrt{\frac{\mu}{L}}\right)^{\frac{8k}{\chi_0}+2}\norm{x_0-x_*}^2.
\end{equation}
Then let us "tweak" the weights of the edges so that the characteristic numbers of their weighted Laplacean numbers increase up to $\chi$. This can always be done by taking any edge and decreasing its weight to $0$, then the smallest positive eigenvalue will go to infinity. Using \eqref{chi_0_chi_relation} and \eqref{chi_0_bound} we obtain
\begin{equation*}
    \norm{x_k-x_*}^2 \ge \left(1-4\sqrt{\frac{\mu}{L}}\right)^{\frac{72k}{7\chi}+2}\norm{x_0-x_*}^2.
\end{equation*}

\section{Conclusion}
In this paper, we study the lower bounds of decentralized optimization problems in the case of a slowly changing communication network. Specifically, we consider the case of changing at most two edges per iteration. A lower bound $\chi\sqrt{L/\mu}\log(1/\epsilon)$ was obtained, which coincides with the lower bound in the class of problems with arbitrary change of edges per iteration. Thus, the question raised in \cite{metelev2023consensus} can be considered closed in the formulation under discussion. However, there are some open questions, such as whether it is possible to obtain acceleration in the case when the graph changes even more slowly (for example, when no more than $\log{n}$ or $const$ edges change per $n$ iterations), or whether it is possible to obtain acceleration on average in the case when changes of the edges are random. 

In the case of a Markovian-varying network, we also obtain a consensus result that allows for the construction of an optimization algorithm whose convergence rate is similar (up to a logarithmic factor) to that of the static case, but with the addition of a term characterized by a Markovian property. Perhaps the extra logarithm could be avoided by constructing a more sophisticated method.

Although the results of the lower bounds are quite pessimistic and argue that no acceleration can be achieved with arbitrary slow changes, some examples (including results in the Markov setting) show that under certain conditions an improvement can be made, and finding such conditions is of interest.

Our work is rooted in a theoretical and mathematical framework, therefore it does not involve the analysis or generation of any datasets.

This work was supported by a grant for research centers in the field of artificial intelligence, provided by the Analytical Center for the Government of the Russian Federation in accordance with the subsidy agreement (agreement identifier 000000D730321P5Q0002) and the agreement with the Moscow Institute of Physics and Technology dated November 1, 2021 No. 70-2021-00138.


\bibliography{references,consensus}
 
\end{document}

%% file: header.tex
\def\<#1,#2>{\langle #1,#2\rangle}
\def\({\left(}
\def\){\right)}

\def\[{\left[}
\def\]{\right]}

\DeclareMathOperator{\spn}{span}

\newcommand{\N}{\mathbb{N}}

\newcommand{\cE}{{\mathcal{E}}}

\newcommand{\cG}{{\mathcal{G}}}
\newcommand{\cH}{{\mathcal{H}}}

\newcommand{\cM}{{\mathcal{M}}}
\newcommand{\cN}{{\mathcal{N}}}

\newcommand{\cV}{{\mathcal{V}}}

\newcommand{\eps}{\varepsilon}

\renewcommand{\epsilon}{\varepsilon}

\newcommand{\ol}{\overline}

\newcommand{\norm}[1]{\left\| #1 \right\|}

\newcommand{\cbraces}[1]{\left( #1 \right)}

\newcommand{\braces}[1]{\left\{ #1 \right\}}

\newcommand{\circledOne}{\text{\ding{172}}}

\def\MKQ{\mathrm{Q}}

\def\taumix{\tau}
\def\nset{\ensuremath{\mathbb{N}}}

\def\nbiter{N}
\def\batchbound{M}

\newcommand{\EE}{\mathbb E}

\usepackage[colorinlistoftodos,bordercolor=blue,backgroundcolor=blue!20,linecolor=blue,textsize=scriptsize]{todonotes}